\documentclass[11pt]{article}
\usepackage{amssymb, amsmath}
\usepackage{amsthm,mathrsfs,subcaption}
\usepackage{hyperref}
\usepackage[margin=1in]{geometry}
 \usepackage{dsfont}
\usepackage{enumerate}
\allowdisplaybreaks[1]
\numberwithin{equation}{section}

\newtheorem{thm}{Theorem}[section]
\newtheorem{cor}{Corollary}[section]
\newtheorem{note}{Note}[section]
\newtheorem{pro}{Proposition}[section]

\newtheorem{defn}{Definition}[section]

\usepackage{graphicx}
\begin{document}
\vspace{-2cm}
\markboth{R. Rajkumar and M.Gayathri}{Spectra of $(H,H')-$merged subdivision graph of a graph}
\title{\LARGE\bf Spectra of $(H_1,H_2)-$merged subdivision graph of a graph}
\author{R. Rajkumar\footnote{e-mail: {\tt rrajmaths@yahoo.co.in}},\ \ \
M. Gayathri\footnote{e-mail: {\tt mgayathri.maths@gmail.com}, }\ \\
{\footnotesize Department of Mathematics, The Gandhigram Rural Institute -- Deemed to be University,}\\ \footnotesize{Gandhigram -- 624 302, Tamil Nadu, India}\\[3mm]
}
\date{ }
\maketitle
\begin{abstract}
In this paper, we define a ternary graph operation which  generalizes the construction of subdivision graph, $R-$graph, central graph. Also, it generalizes the construction of overlay graph (Marius Somodi \emph{et al.}, 2017),  and consequently, $Q-$graph, total graph, and quasitotal graph. We denote this new graph by 
$[S(G)]^{H_1}_{H_2}$, where $G$ is a graph and, $H_1$ and $H_2$ are suitable graphs corresponding to $G$. Further,  we define several new unary graph operations which becomes particular cases of this construction.
We determine the Adjacency and Laplacian spectra of $[S(G)]^{H_1}_{H_2}$ for some classes of graphs $G$, $H_1$ and $H_2$. From these results, we derive the $L$-spectrum of the graphs obtained by the unary graph operations mentioned above.
As applications, these results enable us to compute the number of spanning trees and Kirchhoff index of these graphs. 
\end{abstract}

	\section{Introduction}\label{sec1}

All the graphs considered in this paper are undirected and simple. $K_n$, $C_n$ and $P_n$ denote the complete graph, the cycle graph and the path graph on $n$ vertices, respectively. The complete bipartite graph whose partite sets having sizes  $p$ and $q$ is denoted by $K_{p,q}$.    $J_{n\times m}$ denotes the matrix of size $n\times m$ in which all the entries are 1. We will denote $J_{n\times n}$ simply by $J_n$.

The study of the properties of graphs is an essential one, 	since several real life problems can be modeled by using graphs. One of the approach used in spectral graph theory is, by associating matrices to the given graphs and by determining their eigenvalues and eigenvectors through which the properties of the graphs can be described. 

For a graph $G=(V,E)$ with $V(G) = \{v_1,v_2,\ldots,v_n\}$ and $E(G)=\{e_1,e_2,\ldots,e_m\}$, the \emph{adjacency
	matrix of $G$} is the $n \times n$ matrix $A(G)=[a_{ij}]$, where $a_{ij}=1,$ if $i\neq j$ and, $v_i$ and $v_j$ are adjacent in $G$; 0, otherwise.
The \textit{vertex-edge incidence matrix of $G$} is the $n\times m$ matrix  $B(G)=[b_{ij}]$, where $b_{ij}=1,$ if the vertex $v_i$ is incident with the edge $e_j$; 0, otherwise.
The \textit{degree matrix $D(G)$ of $G$} is the diagonal matrix $diag(d_1,d_2,\ldots,d_n)$, where $d_i$ denotes the degree of the vertex $i$.
The \emph{Laplacian matrix} $L(G)$ of $G$ is the matrix $D(G)-A(G)$ and the \emph{signless Laplacian matrix} $\mathcal Q(G)$ of $G$ is the matrix $D(G)+A(G)$.  Note that $\mathcal Q(G)=B(G)B(G)^T$.
The characteristic polynomials of $A(G)$, $L(G)$ and $\mathcal Q(G)$ are denoted by $P_G(x)$, $L_G(x)$ and $\mathcal Q_G(x)$, respectively. The eigenvalues of $A(G)$, $L(G)$ and $Q(G)$, are said to be the \textit{$A$-spectrum, $L$-spectrum and $\mathcal Q$-spectrum of $G$}, respectively.
Two graphs are said to be \textit{$A$-cospectral} (resp. \textit{$L$-cospectral,  $\mathcal Q$-cospectral}) if they have same the $A$-spectrum (resp. $L$-spectrum, $\mathcal Q$-spectrum). The $A$-specturm, $L$-spectrum and $\mathcal Q$-spectrum of a graph $G$ with $n$ vertices are denoted by $\lambda_i(G),\mu_i(G)$ and $\nu_i(G)$, $i=1,2,\ldots,n,$ respectively.

If $\lambda_1,\lambda_2,\ldots,\lambda_t$ are the distinct eigenvalues of a a matrix $M$ with multiplicity $m_1,m_2,\ldots, m_t$, respectively, then the eigenvalues of $M$ are denoted by $\lambda_1^{m_1},\lambda_2^{m_2}, \ldots, \lambda_t^{m_t}$. If $m_i=1$, for some $i$, then $\lambda_i^{m_i}$ is simply denoted by $\lambda_i$.

Two graphs are said to be \emph{commute} if their adjacency matrices commute. Some properties, examples of commuting graphs are studied in \cite{akbari2007,akbari2009}.
For example, the complete graph $K_n$ commutes with any regular graph of order $n$ and the complete bipartite graph $K_{p,p}$ commutes with any of its regular spanning subgraph (See, \cite[Proposition 2.3.6]{hei2011}).

$A$-spectrum, $L$-spectrum
of a graph are powerful tools for analyzing the properties of the corresponding graph.
Apart from graph theory, the determination of the various spectra of graphs has found applications in many other fields such as physics, chemistry, computer science etc.; see, for instance \cite{bapat2010,brouwer2012, cvetkovic2011}.

Let $G$ be a connected graph with $V(G) = \{ 1,2,\ldots, n\}$. Then the resistance distance $r_{ij}$
between vertices $i$ and $j$ of $G$ is defined to be the effective resistance between nodes $i$ and $j$ as
computed with Ohm's law when all the edges of $G$ are considered to be unit resistors. The \textit{Kirchhoff index} $Kf(G)$ of $G$  is defined as $Kf(G ) = \sum_{i<j}r_{ij}$ \cite{klein1993}.
The resistance distance and
the Kirchhoff index attracted extensive attention due to their wide
applications in electric network theory, physics, chemistry, etc. and the Kirchoff index of graphs constructed by graph operations was also obtained; see \cite{bonchev1994,gao2012,qliu2016, xiao2003, yang2014, zhang2013, zhou2008}.



A natural question arise is "to what extent the spectrum of a given graph can be expressed in terms the spectrum of some other graphs by using graph operations ?".
In this point of view, to construct graphs from the given graphs, several  graph operations were defined  in literature such as the union, the complement, Cartesian product, the Kronecker product, the NEPS, the corona, the edge corona, the join, deletion of a vertex, insertion/deletion of an edge,  etc. and the results on the spectra of the graphs obtrined by using these graph operations were obtained. See, \cite{ barik2007, barik2015, cardoso2013, cui2012,cvetkovic1975, gao2012, hou2010, laali2016, lan2015,liu2014,wang20131} and the references therein. In addition, several unary graph operations were defined in the literature. Some of them are given below for the easy reference of the reader:
The \textit{line graph  $\mathcal{L}(G)$ of $G$} is the graph having $E(G)$ as its vertex set and two vertices are adjacent if and only if the corresponding edges are adjacent in $G$. The  \textit{subdivision graph} $S(G)$ of  $G$ is the graph obtained by inserting a new vertex into every edge of $G$. The \textit{middle graph} or $Q-$\textit{graph} $Q(G)$ of $G$  is the graph obtained from G by inserting a new vertex into each edge of $G$, and joining by edges those pairs of new vertices which lie on adjacent edges of $G$. The \emph{total graph} $T(G)$ of $G$ is the graph obtained by taking one copy of $R(G)$ and joining the new vertices which lie on the adjacent edges of $G$.  The \emph{quasitotal graph $QT(G)$ of $G$} is the graph obtained by taking one copy of $Q(G)$ and joining the vertices which are not adjacent in $G$.  The determination of $A$-spectra of these graphs have been made in \cite{cvetkovic1975,fiedler1973, wang20132,xie2016}. The \emph{central graph  $C(G)$ of $G$} is the graph obtained by taking one copy of $S(G)$ and joining the vertices which are not adjacent in $G$.

In \cite{somody2017}, Marius Somodi \emph{et al.} defined the following graph operation which generalizes
the constructions of the middle, total, and quasitotal graphs:
Let $G$ and $G'$ be two graphs having $n$ vertices with same vertex labeling $\{v_1,v_2,\ldots,v_n\}$. Then the \emph {overlay of $G$ and $G'$}, denoted by $G\ltimes G'$ is the graph obtained by taking $ Q(G)$ and joining the vertices $v_i$ and $v_j$ of $G$ if and only if $v_i$ and $v_j$ are adjacent in $G'$. 
Therein, they obtained the characteristic polynomial of adjacency and Laplacian matrices of overlay of two commuting graphs. Among the other results, they determined the number of spanning trees and Kirchhoff index of overlay of two graphs. Also they derived the $A$-spectrum and $L$-spectrum of $Q-$graph, total graph and quasitotal graph of a graph.

By observing the construction of the above mentioned unary graph operations and overlay graph operation, we define a new ternary graph operation namely, $(H_1,H_2)$-merged subdivision graph of $G$ which is obtained from the subdivision graph of $G$ by combining the suitable graphs $H_1$ and $H_2$. Consequently, this construction generalizes some graph operations defined in the literature, and  enables us to define some new unary operations in Section 2. In Section 3, we obtain the $A$-spectrum and $L$-spectrum of the $(H_1,H_2)$-merged subdivision graph for some classes of graphs $G$, $H_1$ and $H_2$. In addition, we deduce the $L$-spectra of the overlay graph for some class of constituting graphs.  In Section 4, we obtain the number of spanning trees and the Kirchhoff index of  the $(H_1,H_2)$-merged subdivision graph for some classes of graphs $G$, $H_1$ and $H_2$.

\section{$(H_1,H_2)$-merged subdivision graph of $G$}

First we define the following ternary graph operation:
\begin{defn}\label{def1}
	\normalfont
	Let $G$ be a graph with $V(G)=\{v_1,v_2,\ldots,v_n\}$ and $E(G)=\{e_1,e_2,\ldots,e_m\}$. Let $H_1$ and $H_2$ be two  graphs with $V(H_1)=\{u_1,u_2,\ldots,u_n\}$ and $V(H_2)=\{w_1,w_2,\ldots,w_m\}$. The \emph{ $(H_1,H_2)$-merged subdivision graph of $G$}, denoted by $[S(G)]^{H_1}_{H_2}$, is the graph obtained by taking $S(G)$ and joining the vertices $v_i$ and $v_j$ if and only if the vertices $u_i$ and $u_j$ are adjacent in $H_1$, and joining the new vertices which lie on the edges $e_t$ and $e_s$ if and only if $w_t$ and $w_s$ are adjacent in $H_2$ for $i,j=1,2,\ldots,n$ and $t,s=1,2,\ldots,m$. 
\end{defn}


Clearly, if $G$ has $n$ vertices and $m$ edges, and $H_1$ and $H_2$ have $m_1$ and $m_2$ edges, respectively, then $[S(G)]^{H_1}_{H_2}$ has $n+m$ vertices and $2m+m_1+m_2$ edges.

We denote the graphs $[S(G)]_H^{\overline K_n}$ and $[S(G)]_{\overline K_m}^H$  simply by $[S(G)]_H$ and $[S(G)]^H$, respectively. The above construction is illustrated in Figure~\ref{HcombinedRgraphexample}.


%
\begin{figure}[ht]
	\begin{center}
		\includegraphics[scale=1]{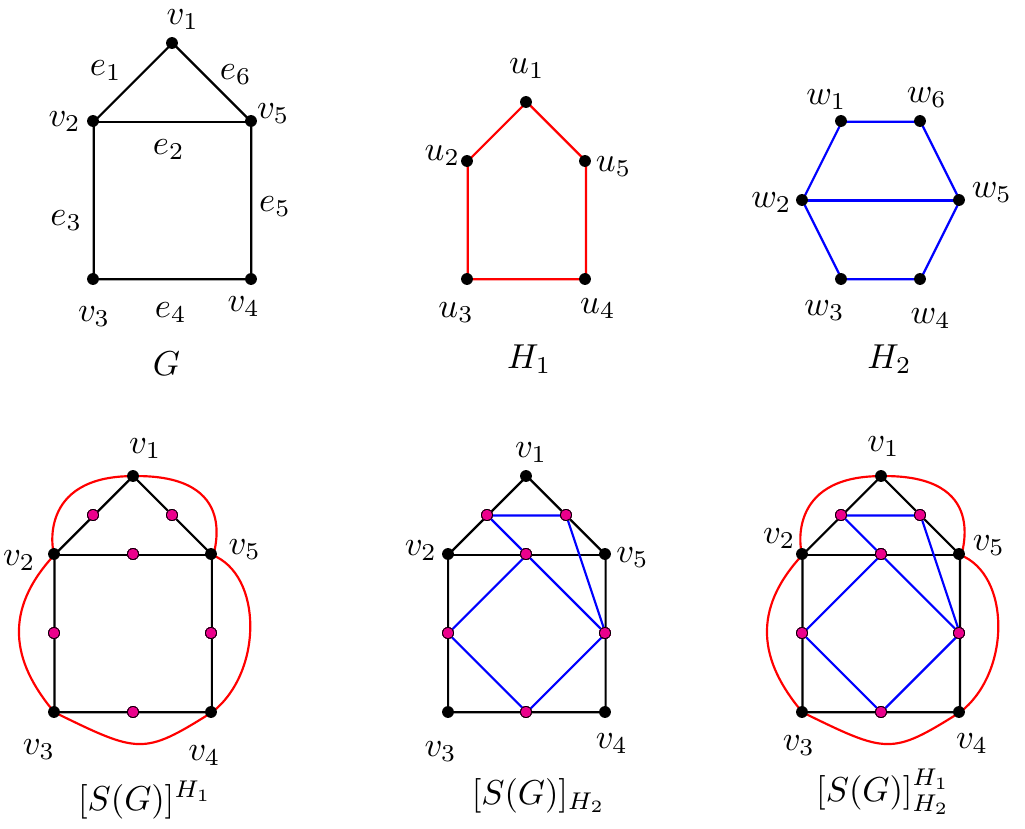}\caption{An example of $(H_1,H_2)$-merged subdivision graph of a graph $G$}\label{HcombinedRgraphexample}
	\end{center}
\end{figure}

The  construction used in Definition~\ref{def1} generalizes
many graph constructions: $S(G) \cong [S(G)]^{\overline K_n}_{\overline K_m}$,  $R(G) \cong [S(G)]^G$ and $C(G) \cong [S(G)]^{\overline G}$. Also note that the graph $[S(G)]^{H}_{\mathcal L(G)}$ is  the graph overlay of $G$ and $H$. Consequently, $Q(G)\cong [S(G)]_{\mathcal L(G)}$, $T(G) \cong [S(G)]^{G}_{\mathcal L(G)}$, $QT(G) \cong [S(G)]^{\overline{ G}}_{\mathcal L(G)}$.

Some of the special cases of Definition~\ref{def1} enable us to define some interesting unary graph operations:

%
\begin{defn}\label{unary combined}
	Let $G$ be a graph with $V(G)=\{v_1,v_2,\ldots,v_n\}$.
	\normalfont
	\begin{enumerate}[(1)]
		\item The \emph{point complete subdivision graph of $G$} is the graph obtained by taking one copy of $S(G)$ and joining all the vertices $v_i,v_j\in V(G)$. 	
		
		
		\item 	The \emph{$Q-$complemented  graph of $G$} is the graph obtained by taking one copy of $S(G)$ and joining the new vertices which lie on the non-adjacent edges of $G$.
		
		\item 	The \emph{total complemented graph of $G$} is the graph obtained by taking one copy of $R(G)$ and joining the new vertices lie which on the non-adjacent edges of $G$.
		
		\item The \emph{quasitotal complemented graph of $G$} is the graph obtained by taking one copy of $Q-$complemented graph of $G$ and joining all the vertices $v_i,v_j\in V(G)$ which are not adjacent in $G$.
		
		\item The \emph{complete $Q-$complemented graph of $G$} is the graph obtained by taking one copy of $Q-$complemented graph of $G$ and joining all the vertices of $v_i,v_j\in V(G)$.
		
		\item The \emph{complete subdivision graph of $G$} is the graph obtained by taking one copy of $S(G)$ and joining the all the new vertices which lie on the edges of $G$.
		
		\item The  \emph{complete $R-$graph of $G$} is the graph obtained by taking one copy of $R(G)$ and joining all the new  vertices which lie on the edges of $G$.
		
		\item 	The  \emph{complete central graph of $G$} is the graph obtained by taking one copy of central graph of $G$ and joining all the new vertices which lie on the edges of $G$.
		
		\item 	The  \emph{fully complete subdivision graph of $G$} is the graph obtained by taking one copy of $S(G)$ and joining all the vertices of $G$ and joining all the new vertices which lie on the edges of $G$.
	\end{enumerate}
\end{defn}

Notice that the graphs mentioned in Definitions~\ref{unary combined}(1)-(9) are isomorphic to $[S(G)]^{K_n}$, $[S(G)]_{\overline{\mathcal L(G)}}$, $[S(G)]_{\overline {\mathcal L(G)}}^{G}$,    $[S(G)]_{\overline { \mathcal L(G)}}^{\overline G}$, $[S(G)]_{\overline {\mathcal L(G)}}^{K_n}$, $[S(G)]_{K_m}$, $[S(G)]_{ K_m}^{G}$, $[S(G)]_{ K_m}^{\overline G}$, $[S(G)]_{ K_m}^{K_n}$, respectively. The structures of these graphs for $G=C_4$  are shown in Figures~\ref{fhhcombinedsubdivisionc4}(a)-(i), respectively.

\begin{figure}[h!]
	\begin{center}
		\includegraphics[scale=0.8 ]{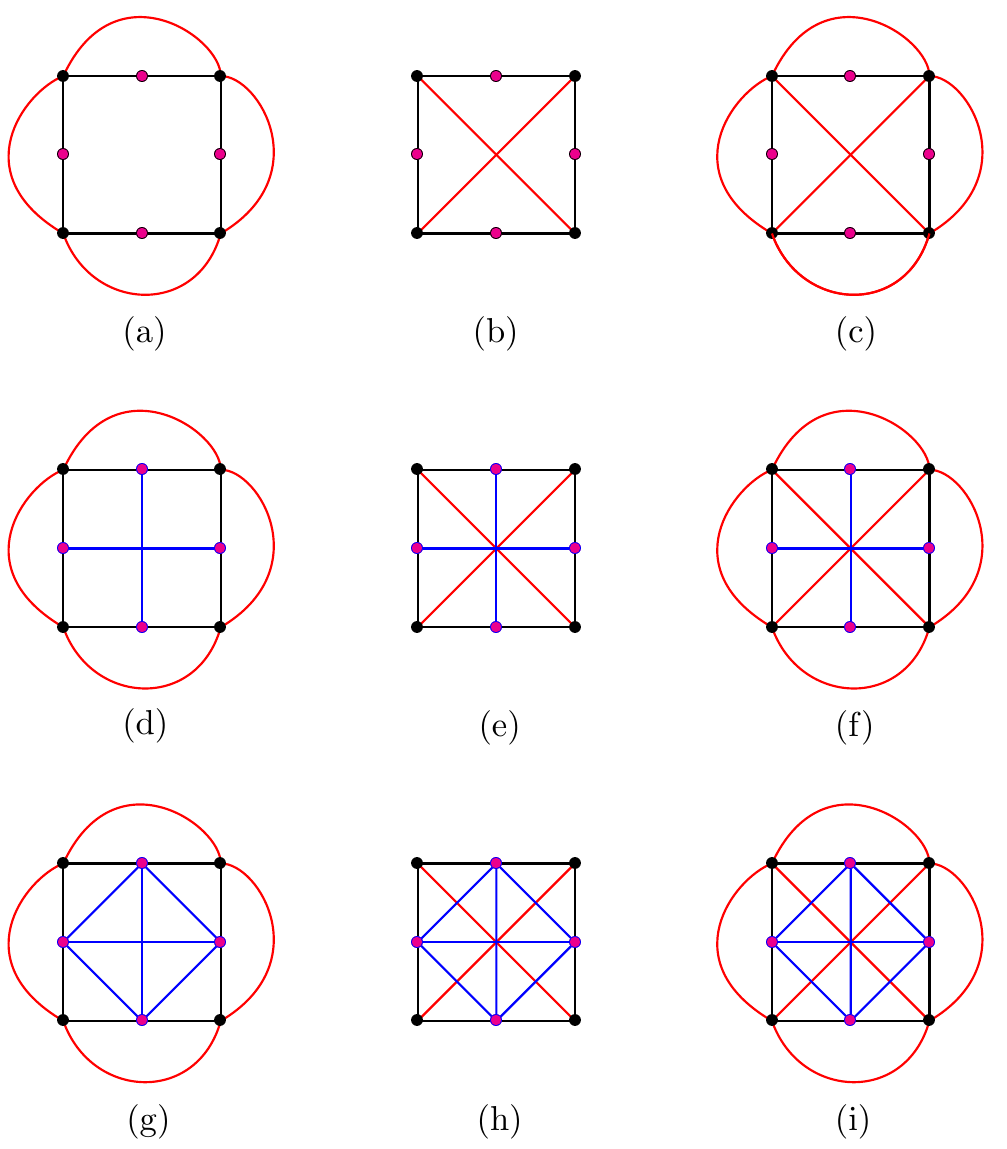}\caption{(a) The point complete subdivision graph of $C_4$, (b) The $Q-$complemented graph of $C_4$, (c) The total complemented graph of $C_4$, (d) The quasitotal complemented graph of $C_4$, (e) The complete $Q-$complemented graph of $C_4$, (f) The complete subdivision graph of $C_4$, (g) The complete $R-$graph of $C_4$, (h) The complete central graph of $C_4$, (i) The fully complete subdivision graph of $C_4$}\label{fhhcombinedsubdivisionc4}
	\end{center}
\end{figure}


\section{$A$-spectra and $L$-spectra of $[S(G)]^{H_1}_{H_2}$}\label{sec2}

In this section, we compute the $L$-spectrum of $[S(G)]^{H_1}_{H_2}$ for some classes of graphs $G$, $H_1$ and $H_2$. 

The Laplacian matrix of $[S(G)]^{H_1}_{H_2}$ is
\begin{eqnarray}\label{lmatrix of h1h2merged subdivision}
	\begin{bmatrix}
		L(H_1)+D(G) & -B(G) \\ -B(G)^T & L(H_2)+2I_m
	\end{bmatrix}
\end{eqnarray}

We first state the following results which will be used later.
\begin{thm}\label{schur}(\cite{bapat2010})
	Let $A$ be an $n\times n$ matrix partitioned as
	\[A=\begin{bmatrix}
	A_1&A_2\\A_3&A_4
	\end{bmatrix},
	\]where $A_1,A_4$ are square matrices. If $A_1,A_4$ are invertible, then	
	\begin{align*}
		\left|A\right|&=\left|A_4\right| \left|A_1-A_2A_4^{-1}A_3\right|=\left|A_1\right| \left|A_4-A_3A_1^{-1}A_2\right|.
	\end{align*}
\end{thm}	

%

%
%
%

To obtain the $L$-spectrum of $[S(G)]^{H}$, $[S(G)]_{K_m}^{H}$  and $[S(G)]_{\overline {\mathcal L(G)}}^{H}$, where $G$ is a regular, first we obtain the characteristic polynomial of a partitioned matrix:
\begin{pro}\label{determinant matrix including JBBT}
	Let $A\in M_n(\mathbb R)$ and $B\in M_{n\times m}(\mathbb R)$. If the sum of all entries in each row of $B$ is equal to $r$, then the characteristic polynomial of the matrix 
	\[M=\begin{bmatrix}
	A & B\\B^T& t_1I_m+t_2J_m+t_3B^TB	
	\end{bmatrix},
	\]  is 
	\[(x-t_1)^{m-n}\times \left|\left\{(x-t_1)I_n-t_3BB^T-\frac{t_2}{2}rJ_n\right\}(xI_n-A)-BB^T\right|.
	\]
\end{pro}

\begin{proof}
	\begin{eqnarray*}
	\begin{vmatrix}
		xI_n-A & -B\\-B^T& (x-t_1)I_m-t_2J_m-t_3B^TB
	\end{vmatrix}  &=&
	\begin{vmatrix}
		xI_n-A & -B\\-B^T-\displaystyle t_3B^T\left(xI_n-A\right)& (x-t_1)I_m-t_2J_{m}
	\end{vmatrix}\\
&&~~~~~~~~~~~~~~~~~~~~~~~~~~~~~~~~~~~~~~~~\left(  R_2\rightarrow R_2-t_3B^TR_1\right) 
	\\ &=&
	\begin{vmatrix}
		xI_n-A & -B\\-B^T-\left\{\displaystyle t_3B^T+\frac{t_2}{2}J_{m\times n}\right\}\left(	xI_n-A\right)& (x-t_1)I_m
	\end{vmatrix}\\
 &&~~~~~~~~~~~~~~~~~~~~~~~~~~~~~~~~~\left( R_2\rightarrow R_2-\frac{t_2}{2}J_{n\times m}R_1\right) .
\end{eqnarray*}
	So, the result follows from Theorem~\ref{schur}.
\end{proof}

The following result gives a characterization of commuting matrices in terms of their eigenvectors. 
\begin{pro}(\cite[Proposition 2.3.2]{hei2011})\label{eigenvectors of commting matrices}
	Let $A_1,A_2,\ldots,A_m$ be symmetric matrices of order $n$. Then the following are equivalent.
	\begin{enumerate}[(1)]
		\item $A_iA_j=A_jA_i$, $\forall i,j\in\{1,2,\ldots,m\}$.
		\item There exists an orthonormal basis $\{x_1,x_2,\ldots, x_n\}$ of $\mathbb R^n$ such that $x_1,x_2,\ldots, x_n$ are eigenvectors of $A_i$, $\forall i=1,2,\ldots,m$.
	\end{enumerate}
\end{pro}


If $G$ and $H$ are two commuting graphs, then by Proposition \ref{eigenvectors of commting matrices}, there exists an orthonormal basis $\{x_1,x_2,\ldots, x_n\}$ of $\mathbb R^n$ such that $x_i'$s are eigenvectors of both $A(G)$ and $A(H)$.   For such graphs, throughout this paper, the $A$-spectra of $G$ and $H$ are denoted by  $\lambda_1(G),\lambda_2(G),\ldots, \lambda_n(G)$ and $\lambda_1(H),\lambda_2(H),\ldots, \lambda_n(H)$, respectively, where $\lambda_i(G)$ and $\lambda_i(H)$ are the eigenvalues of $A(G)$ and $A(H)$ corresponding to the same eigenvector $x_i$, $i=1,2,\ldots,n$. Further, if $G$ and $H$ are  commuting graphs which are $r_1$, $r_2$ regular, respectively, then their $A$-spectra are denoted by $\lambda_1(G)(=r_1),\lambda_2(G),\ldots, \lambda_n(G)$ and $\lambda_1(H)(=r_2),\lambda_2(H),\ldots, \lambda_n(H)$, respectively; their $L$-spectra are denoted by $\mu_1(G)(=0),\mu_2(G),$ $\ldots, \mu_n(G)$ and  $\mu_1(H)(=0),\mu_2(H),\ldots, \mu_n(H)$, respectively. 

\begin{thm}\label{eigenvalues of M}
	Let $G$ be an $r-$regular graph ($r\geq2$) with $n$ vertices and $m~(=\frac{1}{2}nr)$ edges. Let $H$ be a  regular graph with $n$ which commutes with $G$. 
	Let $M$ be a matrix of the form 
	\[M=\begin{bmatrix}
	L(H)+rI_n & B(G) \\ B(G)^T & t_1I_m+t_2J_m+t_3B(G)^TB(G)
	\end{bmatrix}. 
	\]Then the eigenvalues of $M$ are
	\begin{eqnarray*}
		&t_1^{m-n}, \displaystyle\frac{1}{2}\left(r+t_1+2rt_3+mt_2\pm\sqrt{(r-t_1-2rt_3-mt_2)^2+8r} \right), \notag
		\\ &\displaystyle\frac{1}{2}\left( r+\mu_i(H)+t_1+2rt_3-t_3\mu_i(G)\pm\sqrt{\left( r+\mu_i(H)-t_1-2rt_3+t_3\mu_i(G)\right) ^2+4(2r-\mu_i(G))}\right).
	\end{eqnarray*}	
\end{thm}

\begin{proof}
	
	By using Proposition~\ref{determinant matrix including JBBT}, $|xI_{n+m}-M|$ equals
	\begin{eqnarray}
		&(x-t_1)^{m-n}\left|\left\{(x-t_1)I_n-t_3B(G)B(G)^T-\frac{t_2}{2}J_n\right\}\left\{(x-r)I_n-L(H)\right\}-B(G)B(G)^T\right| \nonumber
		\\=&(x-t_1)^{m-n}\left|\left\{(x-t_1)I_n-t_3(2rI_n-L(G))-\frac{t_2}{2}J_n\right\}\left\{(x-r)I_n-L(H)\right\}-2rI_n+L(G)\right|.\notag \\ \label{HHeq2}
	\end{eqnarray}
	Let $\mathcal D=\left|\left\{(x-t_1)I_n-t_3(2rI_n-L(G))-\displaystyle\frac{t_2}{2}J_n\right\}\left\{(x-r)I_n-L(H)\right\}-2rI_n+L(G)\right|$.
	
	For any graph $G'$, the sum of entries in each row (column) of $L(G')$ is 0. So, we have $J_nL(G')=0=L(G')J_n$. That is, for any graph $G'$, $J_n$ commutes with $L(G')$ and so $J_n$ commutes with both $L(G)$ and $L(H)$. Since $G$ and $H$ are regular commuting graphs, $L(G)$ and $L(H)$ also commute. So $J_n$, $L(G)$ and $L(H)$ mutually commute with each other. 
	Then by Proposition \ref{eigenvectors of commting matrices}, there exist orthonormal vectors $x_1,x_2,\ldots,x_n$ which are eigenvectors of $J_n$, $L(G)$ and $L(H)$. 		
	Sine 0 is an eigenvalue of both $L(G)$ and $L(H)$, respectively with eigenvector $\frac{1}{\sqrt{n}}J_{n\times 1}$, we assume that $\mu_1(G)=0$ and $\mu_1(H)=0$. 		
	Since $\frac{1}{\sqrt{n}}J_{n\times 1}$ is an eigenvector of $J_n$ corresponding to the eigenvalue $n$, we take $\lambda_1(J_n)=n$ and all other eigenvalues of $J_n$ are $0$.  
	Let $P$ be the matrix with columns $\frac{1}{\sqrt{n}}J_{n\times 1},x_2,\ldots,x_n$. Then $P$ is orthonormal. Consequently, $P^TL(G)P=diag(0,\mu_2(G),\ldots,\mu_n(G))$, $P^TL(H)P=diag(0,\mu_2(H),\ldots,\mu_n(H))$ and $P^TJ_nP=diag(n,0,0,\ldots,0)$. So, we have
	\begin{eqnarray*}
		\mathcal D&=&|P^T|\left|\left\{(x-t_1)I_n-t_3(2rI_n-L(G))-\frac{t_2}{2}J_n\right\}\left\{(x-r)I_n-L(H)\right\}-2rI_n+L(G)\right||P|
		\\&=&\left|\left\{(x-t_1)I_n-t_3(2rI_n-P^TL(G)P)-\frac{t_2}{2}P^TJ_nP\right\}\left\{(x-r)I_n-P^TL(H)P\right\}-2rI_n+L(G)\right|
		\\&=&\left\{x^2-\left(r+t_1+2rt_3+mt_2\right)x+r(t_1+2rt_3+mt_2)-2r\right\}
		\\&&\times\left\{\prod_{i=2}^{n}\left(x^2-\left[r+\mu_i(H)+t_1+t_3(2r-\mu_i(G))\right]x+t_3(2r-\mu_i(G))(r-\mu_i(H))-2r+\mu_i(G)\right) \right\}.
	\end{eqnarray*}
	Substituting this in \eqref{HHeq2}  we get the result.
\end{proof}
In the following result we deduce the $L$-spectra of some special cases of the graph $[S(G)]^{H_1}_{H_2}$.
\begin{cor}\label{lspectrum of hmerged subdivision}
	Let $G$ be an $r-$regular graph ($r\geq2$) with $n$ vertices and $m~(=\frac{1}{2}nr)$ edges. Let $H$ be a  regular graph with $n$ which commutes with $G$.  Then we have the following.
	\begin{enumerate}[(1)]
		%
		\item The $L$-spectrum of $[S(G)]^{H}$ is
		\[0,r+2, 2^{m-n}, \frac{1}{2}\left( r+\mu_i(H)+2\pm\sqrt{(r+\mu_i(H)+2)^2-8\mu_i(H)-4\mu_i(G)}\right) \text{ for } i=2,3,\ldots,n.
		\]
		
		
		\item  The $L$-spectrum of $[S(G)]_{K_m}^{H}$ is	
		\[0, r+2, (m+2)^{m-n}, \frac{1}{2}\left( m+r+\mu_i(H)+2\pm\sqrt{(m-r-\mu_i(H)+2)^2+4[2r-\mu_i(G)]}\right) \] for $i=2,3,\ldots,n$.

		\item (\cite{somody2017}) The $L$-spectrum of $[S(G)]_{\mathcal L(G)}^{H}$ is	
		\[0, r+2, (2r+2)^{m-n}, \frac{1}{2}\left( r+\mu_i(H)+\mu_i(G)+2\pm\sqrt{(\mu_i(G)-r-\mu_i(H)+2)^2+4[2r-\mu_i(G)]}\right) ,
		\] where $t_i=m-\mu_i(G)+2$ and $i=2,3,\ldots,n.$
		
		\item The $L$-spectrum of $[S(G)]_{\overline{\mathcal L(G)}}^{H}$ is	
		\[0, r+2, (m-2r+2)^{m-n}, \frac{1}{2}\left( t_i+r+\mu_i(H)\pm\sqrt{(t_i-r-\mu_i(H))^2+4[2r-\mu_i(G)]}\right) ,
		\] where $t_i=m-\mu_i(G)+2$ and $i=2,3,\ldots,n.$
	\end{enumerate} 
\end{cor}

\begin{proof}
	
	The Laplacian matrix of $[S(G)]^{H}$, $[S(G)]_{K_m}^{H}$, $[S(G)]_{\mathcal L(G)}^{H}$ and $[S(G)]_{\overline{\mathcal L(G)}}^{H}$ can be obtained by substituting  $L(\overline K_m)=0$, $L(K_m)=mI_m-J_m$, $L(\mathcal L(G))=2rI_m-B(G)^TB(G)$ and $L(\overline {\mathcal L(G)})=(m-2r)I_m-J_m+B(G)^TB(G)$, respectively in \eqref{lmatrix of h1h2merged subdivision}. The $L$-spectrum of these matrices can be obtained, respectively by taking $t_1, t_2$ and $t_3$ in Theorem~\ref{eigenvalues of M} in the following order: (1) $t_1=2$, $t_2=t_3=0$; (2) $t_1=m+2$, $t_2=-1$ and $t_3=0$; (3) $t_1=2r+2$, $t_2=0$ and $t_3=-1$; (4) $t_1=m-2r+2$, $t_2=-1$, and $t_3=1$.
	%
	%
	%
	%
	%
\end{proof}


\begin{note}
	\normalfont
	For a given regular graph $G$, by suitably substituting the graph $H_1$ and $H_2$ in Corollary~\ref{lspectrum of hmerged subdivision}, we can obtain the $L$-spectra of its $R-$graph, central graph and each of the graphs defined in Definitions~\ref{unary combined}(1)-(9). Also, it can be seen that the $L$-spectra of these graphs constructed using $G$ are uniquely determined by the $L$-spectrum of $G$. Consequently, if $G$ and $G'$ are two regular $L$-cospectral graphs, then the graphs constructed  using them as in  Definitions~\ref{unary combined}(1)-(9) are $L$-cospectral.
\end{note}

\noindent\textbf{$(H_1,H_2)$-merged subdivision graph of $K_{p,p}$}

In the next result, we deduce the $L$-spectrum of $(H_1,H_2)$-merged subdivision graph of $K_{p,p}$, for $H_2=\overline K_m,K_m$, $\overline {\mathcal L(K_{p,p})}$. 

To obtain this, we use the following result:

\begin{pro}(\cite[Proof of Lemma 3.13]{bapat2010}) \label{spectrum of bipartite graphs}
	Let $G$ be a bipartite graph with bipartite sets $X$ and $Y$ having $p$ and $q$ vertices, respectively. If $\lambda(G)$ is an eigenvalue of $G$ with eigenvector $[x_1~x_2]^T$, then $-\lambda(G)$ is also an eigenvalue of $G$ with eigenvector $[x_1~-x_2]^T$, where $x_1\in \mathbb R^p$ and $x_2\in \mathbb R^q$.
\end{pro}

\begin{cor}\label{spectra of hhmerged subdivision of Kpp}
	Let $H$ be a spanning $r-$regular subgraph of $K_{p,p}$. Then we have the following.
	\begin{enumerate}[(1)]
		
		\item The $L$-spectrum of $[S(K_{p,p})]^{H}$ is
		\[0,p+2,p+2r, 2^{p^2-2p+1},
		\frac{1}{2}\left( p+\mu_i(H)+2\pm\sqrt{(p+\mu_i(H)+2)^2-8\mu_i(H)-4p}\right)\]
		  $for  i=3,4,\ldots,2p$.

		
		\item The $L$-spectrum of $[S(K_{p,p})]_{K_{p^2}}^{H}$ is	
		\[0, p+2, (p^2+2)^{p^2-2p+1}, p+2r,
		\frac{1}{2} \left( p^2+p+\mu_i(H)+2\pm\sqrt{(p^2-p-\mu_i(H)+2)^2+8p}\right) 	\]   for  $i=3,4,\ldots,2p$.

		
		\item The $L$-spectrum of $[S(K_{p,p})]_{\overline{\mathcal L(K_{p,p})}}^{H}$ is	
		\[0, p+2, (p^2-2p+2)^{p^2-2p+1}, p+2r,
		\frac{1}{2}\left( p^2+\mu_i(H)+2\pm\sqrt{(p^2-2p-\mu_i(H)+2)^2+4p}\right)\]   for  $i=3,4,\ldots,2p$.

		\item The $L$-spectrum of $[S(H)]^{K_{p,p}}$ is
		\[0,r+2, r+2p, 2^{m-2p+1}, 
		\frac{1}{2}\left( r+2+p\pm\sqrt{(r+2+p)^2-8p-4\mu_i(G)}\right)   \text { for } i=3,4,\ldots,2p. 
		\] 
		
		
		\item The $L$-spectrum of $[S(H)]_{K_{m}}^{K_{p,p}}$ is	
		\[0, r+2, r+2p, (m+2)^{m-2p+1},
		\frac{1}{2} \left( m+r+p+2\pm\sqrt{(m-r-p+2)^2+4[2r-\mu_i(G)]}\right) \]  for $i=3,4,\ldots,2p$.

		
		\item The $L$-spectrum of $[S(H)]_{\overline{\mathcal L(H)}}^{K_{p,p}}$ is	
		\[0, r+2, r+2p, (m-2r+2)^{m-2p+1}, \frac{1}{2}\left(s_i+k\pm\sqrt{(s_i-k)^2+4[2r-\mu_i(G)]} \right) ,
		\] where $k=r+p$, $s_i=m-\mu_i(G)+2$ and $i=3,4,\ldots,2p$,
	\end{enumerate}
\end{cor}

\begin{proof}
	Note that the spectrum of $K_{p,p}$ is $p,-p,0^{2p-2}$. Also $J_{2p\times 1}$ and $[J_{1\times p}~ -J_{1\times p}]^T$ are the eigenvectors corresponding to the eigenvalues $p$ and $-p$, respectively. Since $H$ is an $r-$regular spanning subgraph of $K_{p,p}$, it is a $r-$regular bipartite graph. Since $H$ is $r-$regular, $r$ is an eigenvalue of $A(H)$ corresponding to the eigenvector $J_{2p\times 1}$.  By Proposition \ref{spectrum of bipartite graphs}, $-r$ is also an eigenvalue of $H$ corresponding to the eigenvector $[J_{1\times p}~ -J_{1\times p}]^T$, since $H$ is bipartite. 
	Consequently, $0$ is an eigenvalue of $L(K_{p,p})$ (resp. $L(H)$) corresponding to the eigenvector $J_{2p\times 1}$ and $2p$ (resp. $2r$) is an eigenvalue of $L(K_{p,p})$ (resp. $L(H)$) corresponding to the eigenvector $[J_{1\times p}~ -J_{1\times p}]^T$. So, we arrange the $L$-spectrum of $K_{p,p}$ and $H$ such that $\mu_1({K_{p,p}})=0,\mu_2({K_{p,p}})=2p,\mu_3({K_{p,p}})=p,\ldots,\mu_{2p}({K_{p,p}})=p$ and $\mu_1(H)=0,\mu_2(H)=2r,\mu_3(H),\ldots,\mu_{2p}(H)$. Then by using Corollary \ref{lspectrum of hmerged subdivision}, we get the result.
\end{proof}

\begin{note}
	\normalfont
	The argument used in the proof of Corollary~\ref{lspectrum of hmerged subdivision} can be used to derive the adjacency, signless Laplacian spectrum and the normalized Laplacian spectrum of those graphs. 
\end{note}

\noindent\textbf{$(H_1,H_2)$-merged subdivision graph of $K_{1,m}$}
\begin{thm}\label{lchpolygensubstar}
	If $H$ is a graph with $m$ vertices, then we have the following.
	\begin{enumerate}[(1)]
		\item If $H$ is $r-$regular, then the $A$-spectrum of  $[S(K_{1,m})]_{H}$ is \[ 0,\frac{1}{2}\left( r\pm\sqrt{r^2+4m+4}\right) ,\frac{1}{2}\left( \lambda_i(H)\pm\sqrt{\lambda_i(H)^2+4}\right)  \text{ for } i=2,3, \ldots,m.\]
		
		\item The $L$-spectrum of $[S(K_{1,m})]_{H}$ is $$0, \frac{1}{2}\left( m+3\pm\sqrt{(m-1)^2+4}\right) , \frac{1}{2}\left( \mu_i(H)+3\pm\sqrt{[\mu_i(H)+1]^2+4}\right) 
		\text{ for } i=2,3,\ldots,m, $$
	\end{enumerate}
\end{thm}

\begin{proof}
	\begin{enumerate}[(1)]
		\item It is easy to see that
		\begin{equation*}
			A([S(G)]_H)=\begin{bmatrix}
				0&B(K_{1,m})\\B(K_{1,m})^T&A(H)
			\end{bmatrix}.
		\end{equation*}
		By using Theorem~\ref{schur} and the fact $\mathcal L(K_{1,m})=K_m$, we have 
		\begin{eqnarray*}
			P_{[S(K_{1,m})]_{H}}(x)&=&x^{n-m}\left|x^2I_m-xA(H)-B(K_{1,m})^TB(K_{1,m})\right|
			\\&=&x^{n-m}\left|(x^2-2)I_m-xA(H)-A(\mathcal L(K_{1,m}))\right|
			\\&=&x^{n-m}\left|(x^2-2)I_m-xA(H)-A(K_m)\right|
			\\&=&x^{n-m}\left|(x^2-1)I_m-xA(H)-J_m\right|
			\\&=&x(x^2-rx-m-1)\times\left\{\prod_{i=2}^{m}(x^2-\lambda_i(H)x-1)\right\}.
		\end{eqnarray*}
		So the proof follows.
		
		\item 	It is easy to see that
		$$L([S(K_{1,m})]_H)=\begin{bmatrix}
		D(K_{1,m})&-B(K_{1,m})\\-B(K_{1,m})^T&L(H)+2I_m
		\end{bmatrix}.$$
		By using Theorem~\ref{schur}, we have 
		\begin{eqnarray}\label{detlk1m}
			L_{[S(K_{1,m})]_{H}}(x)&=&\left|xI_{m+1}-D(K_{1,m})\right|\times\notag\\ &&\left|xI_m-L(H)-2I_m-B(K_{1,m})\left[xI_{m+1}-D(K_{1,m})\right]^{-1}B(K_{1,m})^T\right|
		\end{eqnarray}
		Since $B(K_{1,m})^T=\begin{bmatrix}
		J_{m\times 1}&I_m
		\end{bmatrix}$ and $D(K_{1,m})=\begin{bmatrix}
		m&0\\0&I_m
		\end{bmatrix}$, we have
		$$\left|xI_{m+1}-D(K_{1,m})\right|=(x-m)(x-1)^m$$ and 
		$$B(K_{1,m})^T(xI-D(K_{1,m}))^{-1}B(K_{1,m})=\frac{1}{x-m}J_m+\frac{1}{x-1}I_m.$$
		Applying these in \eqref{detlk1m}, we get \begin{eqnarray*}
			L_{[S(K_{1,m})]_{H}}(x)&=&(x-m)(x-1)^m~\left|(x-2)I_m-L(H)-\frac{1}{x-m}J_m-\frac{1}{x-1}I_m\right|
			\\&=&(x-m)^{1-m}\left|(x-m)(x-1)\left[(x-2)I_m-L(H)\right]-(x-1)J_m-(x-m)I_m\right|
		\end{eqnarray*}
		Using the eigenvalues of $J_m$,
		\begin{eqnarray*}
			L_{[S(K_{1,m})]_{H}}(x)&=&
			\left\{(x-m)(x-1)(x-2)-m(x-1)-(x-m)\right\}\\
			&&
			\times\left\{\prod_{i=2}^{n}\left((x-1)[x-\mu_i(H)-2]-1\right)\right\}
			\\&=&	x(x^2-(m+3)x+2m+1)\times\left\{\prod_{i=2}^{m}\left (x^2-(\mu_i(H)+3)x+\mu_i(H)+1 \right )\right\}.
		\end{eqnarray*}
		So the proof follows.
	\end{enumerate}
\end{proof}

\noindent\textbf{$(H_1,H_2)$-merged subdivision graph of $P_n$}


The following result is proved in \cite{beezer1984}.
\begin{thm}(\cite[Theorem 3.2]{beezer1984}) \label{pathpolynomial graph}
	Suppose that $p(x)$ is a polynomial of degree less than $n$. Then $p(A(P_n))$ is the adjacency matrix of a graph if and only if $p(x) = P_{P_{2i+1}}(x)$, for some
	$i$, $0\leq i \leq \lfloor{\frac{n}{2}}\rfloor-1$.
\end{thm}

Using this result, we prove the following result.

\begin{cor}
	Let $n \geq 3$ be an integer.	 If $H$ is a graph with $A(H)=P_{P_{2i+1}}(A(P_{n-1}))$, for some $i$, with $0\leq i \leq \lfloor{\frac{n-1}{2}}\rfloor-1$, then the $A$-spectrum of  $[S(P_n)]_H$ is
	\[0, \frac{c_j\pm\sqrt{c_j^2+8\left(\cos\frac{\pi j}{n}+1 \right)}}{2}, \] where $c_j=\displaystyle\sum_{k=0}^{i}(-1)^k\binom{2i+1-k}{k}\left(2\cos\frac{\pi j}{n} \right)^{2(i-k)+1} $ and $j=1,2,\ldots,n-1$.
\end{cor}

\begin{proof}
	It is easy to see that 
	\[A\left([S(P_n)]_H \right)=\begin{bmatrix}
	0 & B(P_n) \\ B(P_n)^T & A(H)
	\end{bmatrix}
	\]
	So by Theorem~\ref{schur}, we have
	\begin{eqnarray*}
		P_{[S(P_n)]_H}(x)&=&x\times\left|x^2I_n-xA(H)-B(P_n)B(P_n)^T\right|
		\\&=&x\times\left|(x^2-2)I_n-xA(H)-A(\mathcal L(P_n))\right|.
	\end{eqnarray*}
	Using the facts $\mathcal L(P_n)=P_{n-1}$ and $A(H)=P_{P_{2i+1}}(A(P_{n-1}))$, we have
	$$	P_{[S(P_n)]_H}(x)=x\times\left\{\displaystyle\prod_{j=1}^{n-1}\left(x^2-P_{P_{2i+1}}(\lambda_j(P_{n-1}))x-\lambda_j(P_{n-1})-2\right)\right\}.$$
	So the proof follows from the facts that $P_{P_{2i+1}}(x)=\displaystyle\sum_{k=0}^{i}(-1)^k\binom{2i+1-k}{k}x^{2(i-k)+1}$ and $\lambda_j(P_{n-1})=2\cos\frac{\pi j}{n}$, where $j=1,2,\ldots,n-1$.
\end{proof}

\noindent\textbf{$Q-$complemented graph of a graph}

For a given $n\times n$ matrix $M$, its coronal $\chi_M(x)$ is defined as $\chi_M(x)=J_{1\times n}(xI_n-M)^{-1}J_{n\times 1}^T$  \cite{liu2014}. For a graph $G$, the coronal of $G$ is defined as the coronal of $A(G)$  and is simply denoted by $\chi_G(x)$.

\begin{pro}(\cite[Proposition 6]{mcleman2011}) \label{equal row sum matrices}
	If $G$ is an $r-$regular graph with $n$ vertices, then \[\chi_G(x)=\frac{n}{x-r}.\]
\end{pro}

\begin{thm}(\cite[Proposition 6]{liu2017}) \label{chpoly of JM}
	Let $M$ be a square matrix of order $n$ and $\alpha$ be a scalar. Then $$\left|xI_n-M-\alpha J_n\right|=\left(1-\alpha \chi_M(x)\right)~|xI_n-M|.$$
\end{thm}

\begin{thm}(\cite[Theorem 2.4.1, Equation (2.28)]{cvetkovic2010}) \label{chpolyline}
	Let $G$ be a graph with $n$ vertices and $m$ edges. Then the characteristic polynomial of $A(\mathcal{L}(G))$  is $$(x+2)^{m-n}\mathcal Q_G(x+2).$$
\end{thm}

\begin{thm}\label{achpoly hcom sub}
	Let $G$ be a graph with $n$ vertices and $m$ edges. Then the characteristic polynomial of $Q-$complemented graph of $G$ is $$(-1)^n(x-1)^m
	\left(1-\frac{x}{1-x}\chi_{\mathcal L(G)}\left(\frac{x^2+x-2}{1-x}\right)\right)Q_G(-x).$$
\end{thm}

\begin{proof}
	Note that 
	\[A\left([S(G)]_{\overline {\mathcal L(G)}} \right) = \begin{bmatrix}
	0 & B(G) \\B(G)^T & A(\overline{\mathcal L(G)})
	\end{bmatrix}
	\]
	So by Theorem~\ref{schur} and using the fact $A(\overline{\mathcal L(G)})=J_m-I_m-A(\mathcal L(G))$
	\begin{eqnarray*}
		P_{[S(G)]_{\overline {\mathcal L(G)}}}(x)&=&x^{n-m}\times\left|x^2I_m-A(\overline{\mathcal L(G)})-B(G)^TB(G)\right|
		\\&=&x^{n-m}\times\left|(x^2+x-2)I_m-xJ_m-(1-x)A(\mathcal L(G))\right|
	\end{eqnarray*}
	By using Theorem \ref{chpoly of JM}, we have,
	\begin{eqnarray*}
		P_{[S(G)]_{\overline {\mathcal L(G)}}(G)}(x)&=&x^{n-m}	\left(1-\frac{x}{1-x}\chi_{\mathcal L(G)}\left(\frac{x^2+x-2}{1-x}\right)\right)\left|(x^2+x-2)I_m-(1-x)A(\mathcal L(G))\right|
		\\&=&x^{n-m} (1-x)^m	\left(1-\frac{x}{1-x}\chi_{\mathcal L(G)}\left(\frac{x^2+x-2}{1-x}\right)\right)P_{\mathcal{L}(G)}\left(\frac{x^2+x-2}{1-x}\right).
	\end{eqnarray*}
	So the proof follows by Theorem \ref{chpolyline}.
\end{proof}

In the following result, we show that for a graph $G$ whose line graph is regular, $Q-$complemented graph of $G$ can be completely determined by the $\mathcal Q$-spectrum of $G$.

\begin{cor}\label{chpolygensubkmlinereg}
	Let $G$ be a graph with $n$ vertices and $m$ edges whose line graph is $r-$regular $(r\geq1)$. Then the $A$-spectrum of $Q-$complemented graph of $G$ is
	$$1^{m-1}, -\nu_i(G), \displaystyle\frac{m-r-1\pm\sqrt{(m-r-1)^2+4r+8}}{2},$$ where $i=2,3,\ldots,n$.
\end{cor}
\begin{proof}
	Since $\mathcal L(G)$ is $r-$regular, by Proposition~\ref{equal row sum matrices} $\chi_{\mathcal L(G)}(x)=\displaystyle \frac{m}{x-r}$. So $$\chi_{\mathcal L(G)}\left(\displaystyle\frac{x^2+x-2}{1-x}\right)=\displaystyle\frac{m(1-x)}{(x+r+2)(x-1)}.$$ By Theorem \ref{achpoly hcom sub},
	the characteristic polynomial of  $Q-$complemented graph of $G$ is 
	\begin{eqnarray}
		P_{[S(G)]_{\overline {\mathcal L(G)}}(G)}(x)&=&(-1)^n(x-1)^m
		\left(1-\frac{mx}{(x+r+2)(x-1)}\right)\mathcal Q_G(-x) \nonumber
		\\&=&(-1)^n(x-1)^m
		\left(\frac{x^2-(m-r-1)x-r-2}{(x+r+2)(x-1)}\right)\mathcal Q_G(-x) \label{HHeqn3}
	\end{eqnarray}
	Also, since $A(\mathcal{L}(G))=B(G)^TB(G)-2I_m$, the sum of the entries in each row of the matrix $B(G)^TB(G)$ is $r+2$ and so $\nu_1(G)=r+2$.
	From this fact, we get $$\mathcal Q_G(-x)=(-1)^n(x+r+2)\times\left\{\displaystyle\prod_{i=2}^{n}(x+\nu_i(G))\right\}.$$  The proof follows by substituting this in \eqref{HHeqn3}. 		
\end{proof}

\begin{cor}
	The $A$-spectrum of $Q-$complemented graph of $K_{p,q}$ is $$0,1^{pq-1},(-p)^{q-1},(-q)^{p-1},\frac{1}{2}\left( pq-p-q+1\pm\sqrt{(pq-p-q+1)^2+4(p+q)}\right).$$
\end{cor}

\begin{proof}
	Since $\mathcal L(K_{p,q})$ is $(p+q-2)-$regular, by using Corollary \ref{chpolygensubkmlinereg} and the $\mathcal Q$-spectrum of $K_{p,q}$ is $p+q, 0, q^{p-1}, p^{q-1} $, we get the $A$-spectrum of $Q-$complemented graph of $K_{p,q}$.
\end{proof}

\noindent\textbf{Complete subdivision graph of a graph}

%

\begin{thm}\label{achpoly hcom sub2}
	
	Let $G$ be a graph with $n$ vertices and $m$ edges. Then the characteristic polynomial of the complete subdivision graph of $G$ is $$(x+1)^{m-n}\left(1-x\chi_{\mathcal L(G)}(x^2+x-2)\right)\mathcal Q_G(x^2+x);$$
\end{thm}

\begin{proof}
	Note that 
	\[A\left([S(G)]_{K_m} \right) = \begin{bmatrix}
	0 & B(G) \\B(G)^T & A(K_m)
	\end{bmatrix}
	\]
	So by Theorem~\ref{schur} and using the fact $A(K_m)=J_m-I_m$, we have
	\begin{eqnarray*}
		P_{[S(G)]_{K_m}}(x)&=&x^{n-m}\left|x^2I_m-x(J_m-I_m)-B(G)^TB(G)\right|
		\\&=&x^{n-m}\left|(x^2+x-2)I_m-xJ_m-A(\mathcal L(G))\right|
		\\&=&x^{n-m}\left(1-x\chi_{\mathcal L(G)}(x^2+x-2)\right)\left|(x^2+x-2)I_m-A(\mathcal L(G))\right|
		\\&=&x^{n-m}\left(1-x\chi_{\mathcal L(G)}(x^2+x-2)\right) P_{\mathcal L(G)}(x^2+x-2).
	\end{eqnarray*}
	
	Proof follows by Theorem \ref{chpolyline}.
\end{proof}

In the next result, we show that for a graph $G$ whose line graph is regular, the $A$-spectrum of complete subdivision graph of $G$ can be completely determined by the $\mathcal Q$-spectrum of $G$.

\begin{cor}\label{chpolygensubkmlinereg2}
	
	\begin{enumerate}[(1)]
		\item  The $A$-spectrum of complete subdivision graph of $tK_{1,2}$ $(t\geq1)$ is
		$$0^t, \displaystyle \left(\frac{-1\pm\sqrt{5}}{2}\right)^t, \left(\frac{-1\pm\sqrt{13}}{2}\right)^{t-1},\frac{1}{2}\left( 2t-1\pm\sqrt{(2t-1)^2+12}\right) $$		
		
		\item Let $G$ be a graph with $n$ vertices and $m$ edges whose line graph is $r-$regular $(r\geq2)$. Then the $A$-spectrum of complete subdivision graph of $G$ is $$(-1)^{m-n},\displaystyle\frac{1}{2}\left( m-1\pm\sqrt{(m-1)^2+4r+8}\right) , \frac{1}{2}\left( -1\pm\sqrt{4\nu_i(G)+1}\right) \text{ for } i=2,3,\ldots,n.$$
	\end{enumerate}
\end{cor}
\begin{proof}
	\begin{enumerate}[(1)]
		\item 	
		Since $\mathcal L(G)$ is $r-$regular, by Proposition~\ref{equal row sum matrices}, $\chi_{\mathcal{L}(G)}(x)=\displaystyle\frac{2t}{x-1}$. Using this fact and the $\mathcal Q$- spectrum of $tK_{1,2}$ is $3^t,1^t,0^t$ in the parts (2) and (3) of Theorem \ref{achpoly hcom sub2}, we get the result.
		
		\item Since $\mathcal L(G)$ is $r-$regular, by Proposition~\ref{equal row sum matrices} $\chi_{\mathcal L(G)}(x)=\displaystyle \frac{m}{x-r}$. So by Theorem \ref{achpoly hcom sub2},
		the characteristic polynomial of $A([S(G)]_{K_m})$ is
		\begin{align}\label{achpoly of complete subdivision line regular}
			(x+1)^{m-n}\left(\frac{x^2-(m-1)x-r-2}{x^2+x-r-2}\right )Q_G(x^2+x).
		\end{align} 
		Also, since $A(\mathcal{L}(G))=B(G)^TB(G)-2I_m$, the sum of the entries in each row of the matrix $B(G)^TB(G)$ is $r+2$ and so $\nu_1(G)=r+2$.
		Using this fact, the proof follows from \eqref{achpoly of complete subdivision line regular}. 
	\end{enumerate}
\end{proof}

\begin{cor}
	Let $(p,q)\neq (1,2), (2,1)$.
	Then the $A$-spectrum of complete subdivision graph of $K_{p,q}$ is	
	$$\displaystyle 0,(-1)^{\alpha},\left(\frac{-1\pm\sqrt{4p+1}}{2}\right)^{q-1},\left(\frac{-1\pm\sqrt{4q+1}}{2}\right)^{p-1},
	\frac{pq-1\pm\sqrt{(pq-1)^2+4(p+q)}}{2},$$	
	where $\alpha = pq-p-q+1$.
\end{cor}

\begin{proof}
	Note that $\mathcal L(K_{p,q})$ is $(p+q-2)-$regular. So by using Corollary \ref{chpolygensubkmlinereg2}(2) and the fact that the $\mathcal Q$-spectrum of $K_{p,q}$ is $p+q, 0, p^{q-1}, q^{p-1} $, we get the $A$-spectrum of  complete subdivision graph of $K_{p,q}$.
\end{proof}

\section{Applications}

In this section, we determine the number of spanning trees and the Kirchhoff index of $[S(G)]_{H_2}^{H_1}$. 
for some families of graphs $G$, $H_1$ and $H_2$.

The number of spanning trees of a graph $G$ is denoted by $\tau(G)$. First, we state a well known result to count the number of spanning tress of a graph using Laplacian eigenvalues.

\begin{thm}(\cite[Theorem 4.11]{bapat2010})\label{spanningtrees}
	Let $G$ be a graph with $n$ vertices. Then
	$$\tau(G)=\mu_2(G)\mu_3(G)\ldots\mu_n(G)/n.$$
\end{thm}

By using Corollary~\ref{lspectrum of hmerged subdivision} and Theorem~\ref{lchpolygensubstar} in Theorem~\ref{spanningtrees}, we have the following result.

\begin{cor}
	Let $G$ be an $r-$regular graph with $V(G)=\{v_1,v_2,\ldots,v_n\}$ and $H$  be a graph of with $V(H)=\{u_1,u_2,\ldots,u_n\}$ which commutes with $G$. Then we have the following.
	\begin{enumerate}[(1)]	
		\item  $\tau\left([S(G)]^{H} \right) =2^{m-n+1}\times \displaystyle\frac{1}{n}
		\left\{\displaystyle\prod_{i=1}^{n}[2\mu_i(H)+\mu_i(G)]\right\}$,
		
		\item  $\tau\left([S(G)]_{K_m}^{H} \right)= (m+2)^{m-n}\times \displaystyle\frac{2}{n}  	\left\{\displaystyle \prod_{i=1}^{n}\left([m+2][r+\mu_i(H)]+\mu_i(G)-2r\right)\right\}$,
		
		\item  $\tau\left([S(G)]_{\overline{\mathcal L(G)}}^{H} \right)= (m-2r+2)^{m-n}\times \displaystyle\frac{2}{n}  	\left\{\displaystyle \prod_{i=1}^{n}\left([r+\mu_i(H)][m-\mu_i(G)+2]+\mu_i(G)-2r\right)\right\}$,
		
		\item If $H$ is a graph with $m$ vertices, then $$\tau([S(K_{1,m})]_{H})=\displaystyle\	\left\{\displaystyle\prod_{i=2}^{m}(\mu_i(H)+1)\right\}.$$	
	\end{enumerate}
\end{cor}

The Kirchhoff index of a connected graph can be calculated by using the following result.
\begin{thm}(\cite[Lemma 3.4]{gao2012}) \label{kf formula}
	For a connected graph $G$ with $n \geq 2$ vertices,
	$$Kf(G)=n\sum_{i=2}^{n}\frac{1}{\mu_i(G)}$$
	
\end{thm}

\begin{cor}
	Let $G$ be an $r-$regular graph with $V(G)=\{v_1,v_2,\ldots,v_n\}$ and $H$  be a graph with $V(H)=\{u_1,u_2,\ldots,u_n\}$ which commutes with $G$. Then we have the following.
	\begin{enumerate}[(1)]	
		\item  $Kf\left([S(G)]^{H} \right) = \displaystyle \frac{n}{2}+\frac{m^2-n^2}{2}+k_1\displaystyle\sum_{i=2}^{n}\left(\frac{k_2+\mu_i(H)}{2\mu_i(H)+\mu_i(G)} \right) $,
		\item  $Kf\left([S(G)]_{K_m}^{H}\right)=\displaystyle \frac{n}{2}+\frac{m^2-n^2}{m+2}+
		k_1\displaystyle\sum_{i=2}^{n}\left(\frac{m+k_2+\mu_i(H)}{m(r+\mu_i(H))+2\mu_i(H)+\mu_i(G)} \right)$,
		\item  $Kf\left([S(G)]_{\overline{\mathcal L(G)}}^{H} \right) =\displaystyle \frac{n}{2}+\frac{m^2-n^2}{k_3}+k_1\displaystyle\sum_{i=2}^{n}\left(\frac{m+k_2+\mu_i(H)-\mu_i(G)}{[m-\mu_i(G)][r+\mu_i(H)]+2\mu_i(H)+\mu_i(G)} \right)$,
		
		\item If $H$ is a graph with $m$ vertices, then  $$Kf([S(K_{1,m})]_H)=m+3+(2m+1)\sum_{j=2}^m\left( \frac{\mu_j(H)+3}{\mu_j(H)+1}\right) .$$
	\end{enumerate}
	where $k_1=m+n$, $k_2=r+2$ and $k_3=m-2r+2$.
\end{cor}

\begin{proof}
	\begin{enumerate}[(1)]
		
		\item Let $a_i=r+\mu_i(H)+2$, $b_i=2\mu_i(H)+\mu_i(G)$. Then by applying Corollary~\ref{lspectrum of hmerged subdivision}(1) in Theorem \ref{kf formula}, we have,
		\begin{eqnarray*}
			Kf\left([S(G)]^{H} \right) &=& \displaystyle \frac{m+n}{r+2}+\frac{m^2-n^2}{2}+(m+n)\displaystyle\sum_{i=2}^{n}\left(\frac{2}{a_i+\sqrt{a_i^2-4b_i}}+\frac{2}{a_i-\sqrt{a_i^2-4b_i}} \right)
			\\&=& \displaystyle \frac{nr+2n}{2(r+2)}+\frac{m^2-n^2}{2}+(m+n)\displaystyle\sum_{i=2}^{n}\left(\frac{a_i}{b_i} \right).
		\end{eqnarray*}

		\item [(2)--(4):] Proof is analogous to the proof of (1) and by using Corollary \ref{lspectrum of hmerged subdivision}(2), \ref{lspectrum of hmerged subdivision}(4) and \ref{lchpolygensubstar}, respectively in Theorem~\ref{kf formula}.			
	\end{enumerate}
\end{proof}

\section {Concluding remarks}

The new graph construction defined in this paper 
generalizes many existing graph operations and enables to define some more new unary graph operations which are
particular cases of the above construction.  From Corollary 3.1, 
the $L$-spectra of the graphs obtained by the existing as well as new unary graph operations mentioned above were
readily derived. 

The determination of spectra of 	$[S(G)]^{H_1}_{H_2}$ for the classes of graphs $G$, $H_1$ and $H_2$ which are not considered in this paper is an interesting problem. Also, the determination of the spectra of other graph matrices of these graphs is further research topic in this direction.  The study of other graph theoretic properties of the graph constructed by this ternary graph operation as well as these new
unary graph operations needs further research.
\section*{Acknowledgment}
The authors would like to thank the referee for his/her useful comments and suggestions. The second author is supported by INSPIRE Fellowship, Ministry of Science and Technology, Government of India under the Grant no. DST/INSPIRE Fellowship/[IF150651] 2015.


\begin{thebibliography}{00}
	
	\bibitem{akbari2007} S. Akbari and A. Herman, Commuting decompositions of complete graphs, {\it J. Comb. Des.} 15 (2) (2007), 133--142.
	
	\bibitem{akbari2009} S. Akbari, F. Moazami and A. Mohammadian, Commutativity of the adjacency matrices of graphs, {\it Discrete Math.} 309 (2009), 595--600.
	
	\bibitem{bapat2010} R. B. Bapat, {\it Graphs and Matrices},  Second Edition, Springer, London, Hindustan Book Agency,
	New Delhi, 2014.
	
	
	
	\bibitem{barik2007}S. Barik, S. Pati, and B. K. Sarma, The spectrum of the corona of two graphs {\it SIAM J. Discrete Math.},  21 (1), (2007), 47--56.
	
	
	\bibitem{barik2015}S. Barik, R. B. Bapat, S. Pati, On the Laplacian spectra of product graphs, {\it Appl. Anal. Discrete Math.} 9 (2015), 39--58.
	
	\bibitem{beezer1984} R. A. Beezer, On the polynomial of a path. {\it Linear Algebra Appl.} 63 (1984), 221--
	225.
	
	\bibitem{brouwer2012} A. E. Brouwer, W. H. Haemers, {\it Spectra of Graphs}, Springer, New York, 2012.
	
	\bibitem{bonchev1994}D. Bonchev, A. T. Balaban, X. Liu and D. J. Klein, Molecular cyclicity and centricity of polycyclic graphs. I: cyclicity based on resistance distances or reciprocal
	distances, {\it Int. J. Quantum Chem.} 50 (1994), 1--20.
	

	\bibitem{cardoso2013}D. M. Cardoso, E. A. Martins, M. Robbiano, O. Rojo, Eigenvalues of a $H-$generalized join graph operation constrained by vertex subsets, {\it Linear Algebra Appl.}, 438 (2013), 3278--3290.
	
	\bibitem{cui2012} S. Y. Cui and G. X. Tian, The spectrum and the signless Laplacian spectrum of coronae, {\it Linear Algebra Appl.} 437 (2012), 1692--1703.
	
	

	
	
	\bibitem{cvetkovic1975}D. Cvetkovi\'c, Spectra of graphs formed by some unary operations, {\it Publ. L'Inst. Math.} 19 (33), (1975), 37--41.
	
		\bibitem{cvetkovic2010}D. Cvetkovi\'c, P. Rowlinson and S. Simi\'c, {\it An Introduction to Theory of Graph Spectra}, Cambridge University Press, New York, 2010.
		
\bibitem{cvetkovic2011}D. Cvetkovi\'c and S. Simi\'c, Graph spectra in computer science, {\it Linear Algebra Appl.} 434 (2011), 1545--1562.	
	
	
	\bibitem{fiedler1973} M. Fiedler, Algebraic connectivity of graphs, {\it Czechoslovak Math. J.} 23(98) (1973), 298--305.
	
	\bibitem{laali2016}A.R. Fiuj Laali, H. Haj Seyyed Javadi and Dariush Kiani, Spectra of generalized corona of graphs, {\it Linear Algebra Appl.}, 493 (2016) 411--425.
	
	\bibitem{gao2012}X. Gao, Y. Luo and W. Liu, Kirchhoff index in line, subdivision and total graphs of a regular graph, {\it Discrete Appl. Math.} 160 (2012), 560--565.
	
	
	
	
	
	
	\bibitem{hei2011}A. Heinze, Applications of Schur rings in algebraic combinatorics: Graphs, partial difference sets and cyclotomic schemes (Ph.D. dissertation),
	Universitat Oldenburg, 2001.
	
	\bibitem{hou2010}Y. Hou and W-C. Shiu, The spectrum of the edge corona of two graphs, {\it Electron J. Linear Algebra} 20(1) (2010), 586--594.
	
	
	
	\bibitem{klein1993}D. J. Klein and M. Randi\'c, Resistance distance, {\it J. Math. Chem.} 12 (1993), 81--95.
	
	\bibitem{lan2015}J. Lan and B. Zhou, Spectra of graph operations based on R-graphs, {\it Linear Multilinear Algebra} 63(7) (2015), 1401--1422.
	
	\bibitem{qliu2016}Q. Liu, Jia-Bao Liu, J. Cao, The Laplacian polynomial and Kirchhoff index of graphs based on R-graphs, {\it Neurocomputing}, 177 (2016), 441--446.
	
	438 (2013), 3547--3559.
	
	\bibitem{liu2014}X. Liu and S. Zhou, Spectra of the neighborhood corona of two graphs, {\it Linear Multilinear Algebra}, 62:9 (2014), 1205--1219
	
	\bibitem{liu2017} X. Liu and Z. Zhang, Spectra of subdivision-vertex and subdivision-edge join of graphs, {\it Bull. Malays. Math. Sci. Soc.}, (2017), 1--17.
	
	
	
	
	\bibitem{mcleman2011}C. McLeman and E. McNicholas, Spectra of coronae, {\it Linear Algebra Appl.} 435 (2011), 998--1007.
	
	
	\bibitem{somody2017}M. Somodi , K. Burke, J. Todd, On a construction using commuting regular graphs, {\it Discrete Math.}, 340 (2017), 532--540.
	
	\bibitem{wang20131} S. L. Wang and B. Shou, The signless Laplacian spectra of the corona and edge corona of two graphs, {\it Linear Multilinear Algebra}, 61(2) (2013), 197--204.
	
	\bibitem{wang20132} W. Wang, D. Yang and Y. Luo, The Laplacian polynomial and Kirchhoff index of graphs derived from regular graphs, {\it Discrete Appl. Math.} 161 (2013), 3063--3071.
	
	\bibitem{xiao2003}W. J. Xiao and I. Gutman, Resistance distance and Laplacian spectrum, {\it Theor. Chem. Acc.} 110 (2003), 284--289.
	
	\bibitem{xie2016}P. Xie, Z. Zhang and F. Comellas, The normalized Laplacian spectrum of subdivisions of a graph, {\it Appl. Math. Comput.} 286 (2016), 250--256.
	
	\bibitem{yang2014} Y. Yang, The Kirchhoff index of subdivisions of graphs, {\it Discrete Appl. Math.}, 171 (2014), 153--157.
	
	\bibitem{zhang2013}Z. Zhang, Some physical and chemical indices of clique-inserted lattices, {\it J. Stat.
		Mech.: Theory Exp.} 10 (2013), 1--12.
	
	
	\bibitem{zhou2008}B. Zhou and N. Trinajsti\'{c}, A note on Kirchhoff index, {\it Chem. Phys. Lett.} 455 (2008), 120--123.
	
	
\end{thebibliography}
\end{document}